\documentclass{birkjour}
\usepackage{amsmath}
\usepackage{xcolor}
\usepackage{mathrsfs}
 \newtheorem{theorem}{Theorem}[section]
 \newtheorem{corollary}[theorem]{Corollary}
 \newtheorem{proposition}[theorem]{Proposition}
 
 \newtheorem{lemma}[theorem]{Lemma}
 \newtheorem{example}[theorem]{Example}
 \theoremstyle{definition}
 \newtheorem{definition}[theorem]{Definition}
 \theoremstyle{remark}
 \newtheorem{remark}[theorem]{Remark}

 \numberwithin{equation}{section}

\begin{document}

%
%
%
%
%
%
%
%
%

 \title[Best approximations in metric spaces with property strongly UC ]{Best approximations in metric spaces with property strongly UC}

 \author[A. Digar]{Abhik Digar}

 \address{%
 Post Doctoral Fellow, 
 Department of Mathematics \& Statistics,
 IIT Kanpur,
 Kalyanpur,
 Uttar Pradesh, India 208 016.}

\email{abhikdigar@gmail.com}




 \begin{abstract}
In this article, we introduce a geometrical notion, property strongly UC which is stronger than property UC and prove the existence of best approximations for a new class of almost cyclic $\psi$-contraction maps in a metric space. As a particular case, we obtain the main results of [Sadiq Basha, S., Best approximation theorems for almost cyclic contractions. J. Fixed Point Theory Appl. 23 (2021), MR4254305] and [Eldred, A. Anthony; Veeramani, P., Existence and convergence of best proximity points. J. Math. Anal. Appl. 323 (2006), MR2260159]. Moreover, we study the existence of best approximations for almost cyclic contractions in a reflexive Banach space.
 \end{abstract}

\maketitle



\section{Introduction and Preliminaries}

An important problem in nonlinear functional analysis and approximation theory is the question of the existence of a best approximation for a nonself map on a given subset of a normed linear space. In other words, if $\mathcal{U}$ is a non-empty subset of a normed linear space $\mathcal{X}$ and $h: \mathcal{U}\to \mathcal{X}$ is a map, then does there exist an element $u\in \mathcal{U}$ such that $\sigma(u,hu)=\sigma(hu, \mathcal{U}):=\inf\{\sigma(hu, z): z\in \mathcal{U}\}?$ Ky Fan's  theorem answers this question affirmatively provided $\mathcal{U}$ is compact convex and $h$ is continuous (\cite{Fan1969}).
 However, lack of sufficient compact sets in metric spaces (in particular, infinite dimensional Banach spaces) makes the existence problem non-trivial.  
 Let $\mathcal{X}^*$ denote the continuous dual of the normed linear space $\mathcal{X}.$ The space $\mathcal{X}$ has the Kadec-Klee property (\cite{book:Megginson}) provided the weak topology and the norm topology coincide on the unit sphere $S_\mathcal{X}=\{x\in \mathcal{X}: \|x\|=1\}.$ If $\mathcal{X}$ is a Banach space, then $\mathcal{X}$ is said to be reflexive if $\mathcal{X}=\mathcal{X}^{**}.$ The normed linear space $\mathcal{X}$ is said to be strictly convex if $S_\mathcal{X}$ does not contain any non-trivial line segment.
The normed linear space $\mathcal{X}$ is said to be uniformly convex (\cite{book:Megginson}) if whenever $\{u_n\}$ and $\{v_n\}$ are sequences in $\mathcal{X}$ such that $\{\|u_n\|\}, \{\|v_n\|\}$ and $\left\{\left\|\frac{u_n+v_n}{2}\right\|\right\}$ converges to $1,$ then $\{\|u_n-v_n\|\}$ converges to $0.$
 
Recently in 2021, S. Basha (\cite{Basha2021}) established the existence of a best approximation (or best approximant) for an almost cyclic contraction map in a uniformly convex Banach space setting. 
  Let $\mathcal{G}$ and $\mathcal{H}$ be two non-empty subsets of a metric space $(\mathcal{E}, \sigma).$ 
 Denote $\{u\in \mathcal{G}: \sigma(u,v)=\sigma(\mathcal{G},\mathcal{H})~\mbox{for some}~v\in \mathcal{H}\}$ and $\{w\in \mathcal{H}: \sigma(z,w)=\sigma(\mathcal{G},\mathcal{H})~\mbox{for some}~z\in \mathcal{G}\}$ by $\mathcal{G}_0$ and $\mathcal{H}_0$ respectively.
The pair $(\mathcal{G},\mathcal{H})$ is proximal if $\mathcal{G}=\mathcal{G}_0$ and $\mathcal{H}=\mathcal{H}_0.$ 
The pair $(\mathcal{G},\mathcal{H})$ is semi-sharp proximal if given an $u$ in $\mathcal{G}$ (resp. in $\mathcal{H}$), there exists at most one $v$ in $\mathcal{H}$ (resp. in $\mathcal{G}$) such that $\sigma(u,v)=\sigma(\mathcal{G},\mathcal{H}).$ The pair $(\mathcal{G},\mathcal{H})$ is sharp proximal if for $u$ in $\mathcal{G}$ (resp. in $\mathcal{H}$), there exists a unique $v$ in $\mathcal{H}$ (resp. in $\mathcal{G}$) such that $\sigma(u,v)=\sigma(\mathcal{G},\mathcal{H}).$
 In this case we write $v=u'.$ We say $(\mathcal{G}, \mathcal{H})$ satisfies some property if $\mathcal{G}$ and $\mathcal{H}$ both have that property.  
 If $(\mathcal{G}, \mathcal{H})$ is a weakly compact convex pair of a Banach space, then $\mathcal{G}_0\neq \emptyset, \mathcal{H}_0\neq \emptyset.$ 
A map $T$ on $\mathcal{G}\cup \mathcal{H}$ is cyclic if $T(\mathcal{G})\subseteq \mathcal{H},~T(\mathcal{H})\subseteq \mathcal{G}$ and a point $u\in \mathcal{G}$ (resp., $u\in \mathcal{H}$) is called a best approximation for $T$ in $\mathcal{G}$ (resp., in $\mathcal{H}$) if $\sigma(u,Tu)=\sigma(Tu, \mathcal{G})$ (resp., $\sigma(u,Tu)=\sigma(Tu, \mathcal{H})$). A huge supply of best approximations for the map $T$ is the best proximity points.  
A point $u\in \mathcal{G}\cup \mathcal{H}$ that satisfies the equation $\sigma(u,Tu)=\sigma(\mathcal{G},\mathcal{H})$ is called a best proximity point for $T$ (\cite{Eldred2006}).
 Here $\sigma(\mathcal{G},\mathcal{H})=\inf\{\|u-v\|: u\in \mathcal{G}, v\in \mathcal{H}\}.$ It is needless to mention that a best proximity point becomes a  fixed point (a point $z\in \mathcal{G}\cup \mathcal{H}$ is a fixed point of $T$ if $z=Tz$) if $\mathcal{G}\cap \mathcal{H}\neq \emptyset.$ 
 Numerous best proximity point theorems for variants of cyclic maps have been one of the main research topics in nonlinear analysis since last two decades. For some references reader may see \cite{Shahzad2009, Abhik2020, Eldred2005, Eldred2006, Espinola2008, EspinolaAurora2011, 
 Aurora2010, Moosa2018, Reich2003, Rajesh2016, Raju2011, Suzuki2009}. The first seminal results in this direction is due to Eldred {\it et. al.} (\cite{Eldred2005,Eldred2006}). 
\begin{definition}
Let $(\mathcal{G}, \mathcal{H})$ be a non-void pair in a metric space and $T$ is a cyclic map on $\mathcal{G}\cup \mathcal{H}$.  Then $T$ is called a cyclic contraction if there exists an $\beta\in [0,1)$ satisfying $\sigma(Tu, Tv)\leq \beta \sigma(u,v)+(1-\beta)\sigma(\mathcal{G},\mathcal{H}).$ 
\end{definition}  
 
 The main theorem of \cite{Eldred2006} is the following: 
\begin{theorem}\label{Thm:Eldred2006}
Let $\mathcal{G}$ and $\mathcal{H}$ be two non-empty closed convex subsets of a uniformly convex Banach space and $T$ a cyclic contraction map on $\mathcal{G}\cup \mathcal{H}.$ Then $T$ has a unique best proximity point. 
\end{theorem}
The class of almost cyclic contraction was introduced in \cite{Basha2021}.
\begin{definition}\label{Def:AlmostCyclic}
Let $\mathcal{G}$ and $\mathcal{H}$ be two nonempty subsets of a metric space and $T$ a cyclic map on $\mathcal{G}\cup \mathcal{H}.$ Then $T$ is said to be an almost cyclic contraction on $\mathcal{G}\cup \mathcal{H}$ if for $u\in \mathcal{G},~v\in \mathcal{H},$ there exists $\beta \in [0,1)$ such that $\sigma(Tu,Tv)\leq \beta \sigma(u,v)+(1-\beta)\sigma(v,\mathcal{G}).$
\end{definition}
The main best approximation theorem in \cite{Basha2021} is the following.
\begin{theorem}\label{Thm:Basha2021}
Suppose that $\mathcal{G}$ and $\mathcal{H}$ are two non-empty subsets of a uniformly convex Banach space with $\mathcal{G}$ is closed convex and $\mathcal{H}$ is convex. Then every almost cyclic contraction on $\mathcal{G}\cup \mathcal{H}$ has a best approximation.
\end{theorem}

 A cyclic map $T$ on $\mathcal{G}\cup \mathcal{H}$ is relatively nonexpansive if $\sigma(Tu,Tv)\leq \sigma(u,v)$ for all $u\in \mathcal{G}$ and $v\in \mathcal{G}$ (\cite{Eldred2005}). It is easy to see that a cyclic contraction is an almost cyclic contraction and an almost cyclic contraction is relatively nonexpansive.
  
The paper is organized as follows. Section \ref{Property SUC} deals with property strongly UC and its properties. In Section \ref{Section3}, we introduce the class of almost cyclic $\psi$-contractions and obtain the existence and uniqueness of best approximations using property strongly UC in a metric space setting. As a consequence, we obtain the main results of \cite{Basha2021} and \cite{Eldred2006}. In Section \ref{Section4}, The existence of best approximations of an almost cyclic contraction in a reflexive Banach space is discussed. In \cite{Basha2021}, the author stated that an almost cyclic contraction is not necessarily continuous. However, we provide sufficient conditions on the domain of a relatively nonexpansive map and discuss its continuity.

\section{Property strongly UC}\label{Property SUC} 
A new geometric notion, property strongly UC is introduced and some properties of this notion are discussed in this section. Let $\mathcal{G}$ and $\mathcal{H}$ be two non-void subsets of a metric space $(\mathcal{E},\sigma).$ It is said that $(\mathcal{G},\mathcal{H})$ has property UC, if whenever $\{u_n\}$ and $\{z_n\}$ are two sequences in $\mathcal{G}$ and $\{v_n\}$ is a sequence in $\mathcal{H}$ with $\sigma(u_n, v_n)\to \sigma(\mathcal{G},\mathcal{H})$ and $\sigma(z_n, v_n)\to \sigma(\mathcal{G},\mathcal{H}),$ then it is the case that $\sigma(u_n, z_n)\to 0.$ For futher study on property UC, one may see \cite{EspinolaAurora2011, Aurora2010, Rajesh2016}.

\begin{definition}\label{strongly UC}
Let $\mathcal{G}$ and $\mathcal{H}$ be two non-empty subsets of a metric space. We say that $(\mathcal{G},\mathcal{H})$ has property strongly UC if $\{u_n\},~\{z_n\}$ are two sequences in $\mathcal{G}$ and $\{v_n\}$ is a sequence in $\mathcal{H}$ satisfying $\sigma(u_n, v_n)-\sigma(v_n, \mathcal{G})\to 0$ and $\sigma(z_n, v_n)-\sigma(v_n,\mathcal{G})\to 0,$ then $\sigma(u_n, z_n)\to 0.$
\end{definition}
We must see that property strongly UC is readily stronger than property UC.
 The following proposition presents that property UC implies property strongly UC under the proximality condition.
\begin{proposition}\label{UCtoSUC}
Let $(\mathcal{G},\mathcal{H})$ be a nonempty proximal pair in a metric space such that $(\mathcal{G},\mathcal{H})$ has property UC. Then $(\mathcal{G},\mathcal{H})$ has property strongly UC.
\end{proposition}
\begin{proof}
Suppose $\{u_n\}, \{z_n\}$ in $\mathcal{G}$ and $\{v_n\}$ in $\mathcal{H}$ are sequences such that $\sigma(u_n,v_n)-\sigma(v_n, \mathcal{G})\to 0$ and $\sigma(z_n,v_n)-\sigma(v_n, \mathcal{G})\to 0.$ Since $(\mathcal{G},\mathcal{H})$ is proximal, 
 there exists $w_n\in \mathcal{G}$ such that $\sigma(w_n,v_n)=\sigma(\mathcal{G},\mathcal{H})$ for $n\geq 1.$ Then
\begin{eqnarray*}
\sigma(u_n,v_n)-\sigma(\mathcal{G},\mathcal{H})&=&\sigma(u_n,v_n)-\sigma(w_n,v_n)\leq \sigma(u_n,v_n)-\sigma(v_n, \mathcal{G})\to 0\\
\sigma(z_n,v_n)-\sigma(\mathcal{G},\mathcal{H})&=&\sigma(z_n,v_n)-\sigma(w_n,v_n)\leq \sigma(z_n,v_n)-\sigma(v_n, \mathcal{G})\to 0.
\end{eqnarray*}
By property UC, we have $\sigma(u_n,z_n)\to 0.$
\end{proof}
For two non-empty closed convex subsets $\mathcal{G}$ and $\mathcal{H}$ of a reflexive Banach space with $\mathcal{G}$ bounded, we have $\mathcal{G}_0\neq \emptyset,~\mathcal{H}_0\neq \emptyset$ (\cite{Reich2003}). Thus as a consequence of Proposition \ref{UCtoSUC}, we have
\begin{corollary}
Suppose $\mathcal{G}$ and $\mathcal{H}$ are two non-empty bounded closed and convex subsets of a reflexive Banach space $\mathcal{E}.$ Then the following are equivalent:\\
\noindent (i) $(\mathcal{G}_0, \mathcal{H}_0)$ has property UC.\\
\noindent (ii) $(\mathcal{G}_0, \mathcal{H}_0)$ has property strongly UC.
\end{corollary}  
The following example illuminates that property UC is in general weaker than property strongly UC. 
\begin{example}\label{UC but not SUC}
Consider the metric space $\mathbb{R}^2$ with the supremum norm and $\mathcal{G}=\{(0,y): 0\leq y\leq 1\},~\mathcal{H}=\{(x,0): 0\leq x\leq 1\}.$ Let $u_n=\left(0, \frac{1}{2}+\frac{1}{n}\right)$ for $n\geq 2,~u_1=(0,0),~ z_n=\left(0, 1-\frac{1}{n}\right)$ and $v_n=\left(1, 0\right)$ for $n\geq 1.$ 
 We see that $\|u_n-v_n\|-\sigma(v_n,\mathcal{G})\to 0$ and $\|z_n-v_n\|-\sigma(v_n,\mathcal{G})\to 0.$ But $\|u_n-z_n\|\to \frac{1}{2}.$ Therefore, $(\mathcal{G},\mathcal{H})$ does not have property strongly UC. Notice that $(\mathcal{G},\mathcal{H})$ has property UC.
\end{example} 
\begin{proposition}
Let $\mathcal{G}$ and $\mathcal{H}$ be two nonempty subsets of a strictly convex Banach space $\mathcal{E}.$ Assume that $\mathcal{E}$ has the Kadec- Klee property, $\mathcal{G}$ is weakly compact and convex, and $\mathcal{H}$ is compact. Then $(\mathcal{G},\mathcal{H})$ has property strongly UC.
\end{proposition}
\begin{proof}
Let us consider the sequences $\{u_n\},~\{z_n\}$ in $\mathcal{G}$ and $\{v_n\}$ in $\mathcal{H}$ such that $\|u_n-v_n\|-\sigma(v_n, \mathcal{G})\to 0$ and $\|z_n-v_n\|-\sigma(v_n, \mathcal{G})\to 0.$ If for any subsequence $\{n_j\},$ $\sigma(v_{n_j}, \mathcal{G})\to 0,$ then by triangle inequality we have $\|u_{n_j}-z_{n_j}\|\to 0.$ 
Since $\mathcal{H}$ is weakly compact, it is bounded and hence is the sequence $\{\sigma(v_n, \mathcal{G})\}.$ Without loss of any generality, assume that $\sigma(v_n, \mathcal{G})\to P$ for some $P>0.$ Then $\|u_n-v_n\|\to P$ and $\|z_n-v_n\|\to P.$ There exists a subsequence $\{n_j\}$ satisfyingt $u_{n_j}\xrightarrow{w} u,~ z_{n_j}\xrightarrow{w} z$ and $v_{n_j}\rightarrow v.$ Then $P=\displaystyle\lim_{n \to \infty}\sigma(v_{n_j}, \mathcal{G})=\sigma(v, \mathcal{G}).$ 
By weakly lower semicontinuous of the norm, we have
\begin{eqnarray*}
P=\sigma(v, \mathcal{G})\leq \|u-v\|\leq \liminf_{j \to \infty} \|u_{n_j}- v_{n_j}\|\leq \limsup_{j \to \infty} \|u_{n_j}- v_{n_j}\| =P.
\end{eqnarray*}
Thus $ \|u-v\|=P.$ Similarly, $\|z-v\|=P.$ By the Kadec- Klee property, we have $u_{n_j}- v_{n_j} \to u-v$ and $z_{n_j}- v_{n_j} \to z-v.$ 
 If $u\neq z,$ by strict convexity, we get
\begin{eqnarray*}
\sigma(v, \mathcal{G})\leq \left\|v-\frac{u+z}{2}\right\|=\left\|\frac{u-v}{2}+\frac{z-v}{2}\right\|<P=\sigma(v, \mathcal{G}),~\mbox{which is absurd.}
\end{eqnarray*}
Thus $u=z$ and this completes the proof.
\end{proof}
A rich supply of pairs that possess property strongly UC comes from the collection of bounded and convex subsets of uniformly convex Banach spaces.
\begin{proposition}
Let $\mathcal{G}$ and $\mathcal{H}$ be two nonempty bounded subsets of a uniformly convex Banach space $\mathcal{E}$ with $\mathcal{G}$ convex. Then $(\mathcal{G},\mathcal{H})$ has property strongly UC. 
\end{proposition}
\begin{proof} 
Suppose $\{u_n\}$ and $\{z_n\}$ are sequences in $\mathcal{G}$ and $\{v_n\}$ is a sequence in $\mathcal{H}$ such that $\|u_n-v_n\|-\sigma(v_n, \mathcal{G})\to 0$ and $\|z_n-v_n\|-\sigma(v_n,\mathcal{G})\to 0.$ 
Assume that $(\mathcal{G},\mathcal{H})$ does not have property strongly UC. Then there exists $\epsilon_0>0$ and a subsequence $\{n_j\}$ of $\{n\}$ such that 
$\|u_{n_j}-z_{n_j}\|\geq \epsilon_0$ for each $j\geq 1.$ In the case, when $\sigma(v_{n_j}, \mathcal{G})\to 0$ for some $\{n_j\},$ then $\|u_{n_j}-y_{n_j}\|\to 0$ and $\|z_{n_j}-v_{n_j}\|\to 0.$ Hence by triangle inequality, we have $\|u_{n_j}-z_{n_j}\|\to 0,$ a contradiction. Passing via subsequences, if necessary, we may assume that $\{\sigma(v_{n_j}, \mathcal{H})\}$ converges, say to $d>0.$ Then $\|u_{n_j}-v_{n_j}\|\to d,~ \|z_{n_j}-v_{n_j}\|\to d$ and $\left\|\frac{(u_{n_j}-v_{n_j})+(z_{n_j}-v_{n_j})}{2} \right\|=\left\|v_{n_j}-\frac{u_{n_j}+z_{n_j}}{2} \right\|\geq \sigma(v_{n_j}, \mathcal{G})\to d.$ Thus $\left\|\frac{(u_{n_j}-v_{n_j})+(z_{n_j}-v_{n_j})}{2} \right\|\to d.$ By uniform convexity of $\mathcal{E}$, it is clear that $\|u_{n_j}-z_{n_j}\|=\|(u_{n_j}-v_{n_j})-(z_{n_j}-v_{n_j})\|\to 0.$ This is a contradiction and hence the proof is complete.
\end{proof}
Thus, we observe that there is no geometrical distinction between property UC and property strongly UC of a bounded pair $(\mathcal{G},\mathcal{H})$ with $\mathcal{G}$ convex in a uniformly convex Banach space setting. 
The following result is easy to observe.
\begin{proposition}\label{Proposition_semi}
Let $\mathcal{G}$ and $\mathcal{H}$ be two non-empty subsets of a metric space $(\mathcal{E},\sigma).$ If $(\mathcal{G},\mathcal{H})$ has property strongly UC, then $(\mathcal{G},\mathcal{H})$ is semi-sharp proximal. 
\end{proposition}
Affine hull of a non-empty subset $\mathcal{C}$ of a real normed linear space $\mathcal{X}$ is defined as aff$(\mathcal{C})=\{tx+(1-t)y: x, y\in \mathcal{C},~t~\mbox{is real}\}.$ If $\mathcal{C}$ is convex, then it can be easily verified that $\mathcal{V}=\mathcal{C}-x_0$ is a vector subspace for any $x_0\in \mathcal{C}.$ In this case, the dimension of the subspace $\mathcal{V}$ is considered as the dimension of the convex set $\mathcal{C}.$ The following result provides a sufficient condition for the strict convexity of a Banach space.
\begin{theorem}
Let $\mathcal{E}$ be a Banach space. Consider the following statements:\\
\noindent (i) $\mathcal{E}$ is strictly convex.\\
\noindent (ii) If $\mathcal{G}$ is any one dimensional compact convex subset of the unit sphere of $\mathcal{E},$ then $(\mathcal{G}, \mathcal{H})$ has property strongly UC, where $\mathcal{H}$ is the zero space.\\
\noindent Then $(ii)\Rightarrow (i).$
\end{theorem}
\begin{proof}
On the contrary assume that $\mathcal{E}$ is not strictly convex. Choose $w_1, w_2\in S_\mathcal{E}$ and $\delta>0$ such that $\|w_1- w_2\|>\delta$ and  $[w_1, w_2]:=\{tw_1+(1-t)w_2: t\in [0,1]\}\subseteq S_\mathcal{E}.$ 
 Set $\mathcal{G}=[w_1, w_2],~\mathcal{H}=\{0\}.$ Consider $u_n=w_1, z_n=w_2, n\geq 1.$ For any sequence $\{v_n\}$ in $\mathcal{H}$ (which is the sequence of $0$'s), we have $\|u_n-v_n\|-\sigma(v_n, \mathcal{G})=\|w_1\|-1= 0$ and $\|z_n-v_n\|-\sigma(v_n, \mathcal{G})=\|w_2\|-1= 0.$ But $\|u_n-z_n\|>\delta$ for all $n.$
\end{proof}
The following lemma is crucial in showing the existence of best approximation points. The techniques similar to Lemma 2 of \cite{Suzuki2009} can be used to prove it.
\begin{lemma}\label{Cauchy}
Let $\mathcal{G}$ and $\mathcal{H}$ be two non-void subsets of a metric space $(\mathcal{E},\sigma).$ Assume that $(\mathcal{G},\mathcal{H})$ has property strongly UC. Suppose $\{u_k\}$ is a sequence in $\mathcal{G}$ and $\{v_k\}$ a sequence in $\mathcal{H}$ such that one of the following holds:\\
$\displaystyle \lim_{k\to \infty} \sup_{l\geq k} \left[\sigma(u_l, v_k)-\sigma(v_k, \mathcal{G})\right]=0,$ or $\displaystyle \lim_{l\to \infty} \sup_{k\geq l} \left[\sigma(u_l, v_k)-\sigma(v_k, \mathcal{G})\right]=0.$ Then $\{u_k\}$ is Cauchy.
\end{lemma}

\section{Best approximations for almost cyclic $\psi$-contractions}\label{Section3}
 In this section, we introduce the class of almost cyclic $\psi$-contraction maps and establish the existence of best approximation theorems using property strongly UC.
\begin{definition}
Let $\mathcal{G}$ and $\mathcal{H}$ be two non-empty subsets of a metric space $(\mathcal{E},\sigma)$ and let $\psi: [0,\infty) \to [0,\infty)$ be a strictly increasing and continous map. A cyclic map $T:\mathcal{G}\cup \mathcal{H}\to \mathcal{G}\cup \mathcal{H}$ is said to be an almost cyclic $\psi$-contraction if 
\begin{eqnarray*}
\sigma(Tx,Ty)&\leq & (I-\psi)(\sigma(x,y))+\psi(\sigma(y,\mathcal{G}))~\mbox{for all}~ (x,y)\in \mathcal{G}\times \mathcal{H}
\end{eqnarray*}
where $I$ is the identity map on $[0,\infty).$
\end{definition}
A few examples are in order.
\begin{example}\label{ex1}
A cyclic contraction is an almost cyclic $\psi$-contraction for $\psi (s)=(1-\beta)s,~\beta\in [0,1).$
\end{example}
\begin{example}\label{ex2}
An almost cyclic contraction is an almost cyclic $\psi$-contraction for $\psi (s)=(1-\beta)s,~\beta\in [0,1).$
\end{example}
However, an almost cyclic $\psi$-contraction is not necessarily an almost cyclic contraction.
\begin{example}
Suppose $\mathcal{E}=\{h: [0,1]\to \mathbb{C}: h~\mbox{is continuous}\}$ with $\|h\|=\max\{\|h_1\|_\infty, \|h_2\|_\infty\},$ where $h\in \mathcal{E},~h=h_1+i h_2$ and for $1\leq j\leq 2,~h_j: [0,1]\to \mathbb{R}$ is continous.  Set $\mathcal{G}=\{f_1+if_2\in \mathcal{E}: 0\leq f_1(s)\leq 1, f_2(s)=0~\mbox{for}~s\in [0,1]\}$ and $\mathcal{H}=\{g_1+ig_2\in \mathcal{E}: g_1(s)=0,~0\leq g_2(s)\leq 1 ~\mbox{for}~s\in [0,1]\}.$ Suppose $\psi:[0, \infty)\to [0,\infty)$ is given by $\psi(t)=\frac{t^2}{1+t}.$ Define $T:\mathcal{G}\cup \mathcal{H}\to \mathcal{G}\cup \mathcal{H}$ by
\[
 	Tf=
 	\begin{cases}
		f_2+i  \left(\displaystyle\frac{1}{1+\|f_1\|_\infty}\right) f_1 ~&\text{if}~f=f_1+if_2\in \mathcal{G}; \\
 		\left(\displaystyle\frac{1}{1+\|f_2\|_\infty}\right)f_2+ if_1   ~&\text{if}~f=f_1+if_2\in \mathcal{H}.
	\end{cases}
 \] 

Then $T$ is an almost cyclic $\psi$-contraction. Moreover, for any $\beta\in [0,1)$ and chosen an $f_1\in \mathcal{G}$ with $0<\|f_1\|_\infty<\frac{1-\beta}{1+\beta},$ we have $\left\|T(f_1)-T(\frac{if_1}{2})\right\|=\frac{\|f_1\|_\infty}{1+\|f_1\|_\infty}> \beta \|f_1\|_\infty+(1-\beta)\frac{\|f_1\|_\infty}{2}.$  
Thus $T$ is not an almost cyclic contration.
\end{example}
 Let $\psi:[0,\infty) \to [0,\infty)$ be a strictly increasing function. A cyclic map $T$ is a cyclic $\psi$-contraction (\cite{Shahzad2009}) on $\mathcal{G}\cup \mathcal{H}$ provided $\sigma(Tx,Ty)\leq (I-\psi)(\sigma(x,y))+\psi(\sigma(\mathcal{G},\mathcal{H}))$ for all $x\in \mathcal{G}, y\in \mathcal{H}.$ It is clear that every cyclic $\psi$-contraction, when $\psi$ is continuous, is an almost cyclic $\psi$-contraction. The below example establishes that every almost cyclic $\psi$-contraction is not cyclic $\psi$-contraction.
\begin{example}\label{ex3}
Let $\mathcal{E}=\mathbb{R}.$ Suppose $\mathcal{G}= [1, 2]$ and $\mathcal{H}=[-2, -1].$ Define $T:\mathcal{G}\cup \mathcal{H}\to \mathcal{G}\cup \mathcal{H}$ by $T(x)=-1$ for $x\in \mathcal{G},~T(y)=-y$ for $y\in \mathcal{H}.$ Then $\sigma(\mathcal{G},\mathcal{H})=2$ and $\sigma\left(y, \mathcal{G}\right)=1-y$ for $y\in \mathcal{H}.$ Let $\psi:[0,\infty)\to [0,\infty)$ be defined by $\psi(s)=s+1,~s\geq 0.$ Then $T$ is an almost cyclic $\psi$-contraction. But this is not cyclic $\psi$-contraction. Since $|Tx-T(-2)|=3>2=|x-(-2)|-(|x-(-2)|+1) +(2+1)=\sigma(x, -2)-\psi(\sigma(x, -2))+\psi(\sigma(\mathcal{G},\mathcal{H}))$ for any $x\in \mathcal{G}.$
\end{example}
With the above preparation, we now turn to the existence of best approximations for almost cyclic $\psi$-contractions.   
\begin{theorem}\label{cgtsub}
Let $\mathcal{G}$ and $\mathcal{H}$ be two non-empty bounded subsets of a metric space $(\mathcal{E},\sigma).$ 
 Let $T$ be an almost cyclic $\psi$-contraction on $\mathcal{G}\cup \mathcal{H}.$ Let $u_0\in \mathcal{G}$ and define $u_{n+1}=Tu_n,~n\geq 0.$ If $\{u_{2n}\}$ has a subsequence that converges in $\mathcal{G},$ then $T$ has a best approximation.
\end{theorem}
\begin{proof}
First notice that $\sigma(u_{2n+1}, u_{2n+2})\leq (I-\psi)(\sigma(u_{2n}, u_{2n+1}))+\psi(\sigma(u_{2n+1}, \mathcal{G}))$ $\leq \sigma(u_{2n}, u_{2n+1})$ for $n\geq 0.$ Thus the sequence $\{\sigma(u_{2n+1}, u_{2n+2})\}$ is nonincreasing and $0\leq \sigma(u_{2n+1}, u_{2n+2})\leq \sigma(u_0, x_1)$ for $n\geq 0.$ Let $\sigma_0= \displaystyle\lim_{n\to \infty}\sigma(u_{2n}, u_{2n+1}).$
Choose a subsequence $\{u_{2n_k}\}$ that converges, say to $z$ in $\mathcal{G}.$ Then
$\sigma(u_{2n_k+2},Tz)\leq  \sigma(u_{2n_k+1,}, z)
\leq  \sigma(u_{2n_k,}, u_{2n_k+1}) +\sigma(u_{2n_k,}, z).$
As $k\to \infty,$ we have 
\begin{align}\label{inequality1}
\sigma(z,Tz)\leq \sigma_0.
\end{align}
Set $L=\sup\{\sigma(x,y):x\in \mathcal{G}, y\in \mathcal{H}\}.$
Since $\psi$ is continuous and strictly increasing on $[0,L]$, we must say $\psi^{-1}$ exists and $\psi^{-1}$ is strictly increasing on $\psi ([0,L]).$ 
Now
\begin{eqnarray*}
\sigma\left(u_{2{n_k}+2}, u_{2{n_k}+1}\right) &\leq & (I-\psi)(\sigma\left(u_{2{n_k+1}}, u_{2{n_k}}\right)+\psi(\sigma\left(u_{2{n_k}+1}, \mathcal{G}\right))\\
&\leq & (I-\psi)(\sigma\left(u_{2{n_k}+1}, u_{2{n_k}}\right)+\psi(\sigma\left(u_{2{n_k}+1}, T^2z\right))\\
&\leq & (I-\psi)(\sigma\left(u_{2{n_k}+1}, u_{2{n_k}}\right)+ \psi\big[(I-\psi)(\sigma\left(u_{2{n_k}}, Tz\right)\\
&& +\psi(\sigma(Tz, \mathcal{G}))\big].
\end{eqnarray*}
Letting $k\to \infty$ and using the continuity of $\psi,$ we have 
\begin{eqnarray*}
\sigma_0\leq (I-\psi)(\sigma_0)+\psi\left[(I-\psi)(\sigma\left(z, Tz\right)+\psi(\sigma(Tz, \mathcal{G}))\right]
\end{eqnarray*}
Thus, $\psi(\sigma_0)\leq \psi\left[(I-\psi)(\sigma\left(z, Tz\right))+\psi(\sigma(Tz, \mathcal{G}))\right].$ As $\psi^{-1}$ is strictly increasing on $\psi ([0,L]),$ using (\ref{inequality1}) we get
\begin{eqnarray*}
\sigma_0\leq (I-\psi)(\sigma\left(z, Tz\right))+\psi(\sigma(Tz, \mathcal{G}))\leq (I-\psi)(\sigma_0)+\psi(\sigma(Tz, \mathcal{G})).
\end{eqnarray*}
Then $\psi(\sigma_0)\leq \psi(\sigma(Tz, \mathcal{G})).$ 
 Using the inequality (\ref{inequality1}) and the fact that $\psi$ is strictly increasing, we have $\sigma(z,Tz)=\sigma(Tz,\mathcal{G}).$ 
\end{proof}
Let $\mathcal{G}, \mathcal{H},T$ and $\{u_n\}$ as taken in Theorem \ref{cgtsub}. Using relatively nonexpansiveness of $T,$ we have $\sigma(u_{2n+1}, T^2z)\leq \sigma(u_{2n}, Tz)\leq \sigma(u_{2n-1}, z)$ for any $z\in \mathcal{G}.$ Consequently, $\sigma(u_{2n+1}, \mathcal{G})\leq \sigma(u_{2n-1}, \mathcal{G})$ for all $n\geq 1.$
\begin{lemma}\label{Approximationlemma}
Suppose $\mathcal{G}$ and $\mathcal{H}$ are non-empty subsets of a metric space $(\mathcal{E},\sigma)$ and $T$ is an almost cyclic $\psi$-contraction on $\mathcal{G}\cup \mathcal{H}.$ For $u_0\in \mathcal{G},$ define $u_{n+1}=Tu_n,~n\geq 0.$ Then $\sigma(u_{2n+2}, u_{2n+1})-\sigma(u_{2n+1}, \mathcal{G})\to 0$ and $\sigma(u_{2n}, u_{2n+1})-\sigma(u_{2n+1}, \mathcal{G})\to 0.$
\end{lemma}
\begin{proof}
 Being non-increasing and bounded, the sequence $\{\sigma(u_{2n+1}, \mathcal{G})\}$ is convergent. Let $r:=\displaystyle\lim_{n\to \infty}\sigma(u_{2n+1}, \mathcal{G}).$
From the proof of Theorem \ref{cgtsub}, the sequence $\left\{\sigma(u_{2n}, u_{2n+1})\right\}$ converges as well. Let $\sigma_0:=\displaystyle\lim_{n\to \infty}\sigma(u_{2n}, u_{2n+1}).$ Clearly, $r\leq \sigma_0.$
Since
$\sigma(u_{2n}, u_{2n+1})\leq (I-\psi) (\sigma(u_{2n}, u_{2n-1})) +\psi (\sigma(u_{2n-1}, \mathcal{G}))\leq \sigma(u_{2n}, u_{2n-1})~\mbox{for}~n\geq 1,$
by continuity of $\psi,$ we get $\psi(r)=\psi(\sigma_0).$ As $\psi$ is strictly increasing, we have $r=\sigma_0.$ Therefore, $\sigma(u_{2n+2}, u_{2n+1})-\sigma(u_{2n+1}, \mathcal{G})\to 0$ and similarly $\sigma(u_{2n}, u_{2n+1})-\sigma(u_{2n+1},\mathcal{G})\to 0.$
\end{proof}
We now prove the best approximation theorem for an almost cyclic $\psi$-contraction in the metric space setting.
\begin{theorem}\label{maintheorem}
Suppose $\mathcal{G}$ and $\mathcal{H}$ are two non-empty subsets of a metric space $(\mathcal{E},\sigma)$ such that $\mathcal{G}$ is complete and $(\mathcal{G},\mathcal{H})$ has property strongly UC. Let $T$ be an almost cyclic $\psi$-contraction on $\mathcal{G}\cup \mathcal{H}.$
 Then $T$ has a unique best approximation in $\mathcal{G}.$ Moreover, if $u_0\in \mathcal{G}$ and $u_{n+1}=Tu_n,~n\geq 0,$ then $\{u_{2n}\}$ converges to the best approximation in $\mathcal{G}.$
\end{theorem}
\begin{proof}
We first prove that for each $\epsilon>0,$ there exists $N_0\in \mathbb{N}$ such that 
\begin{eqnarray}\label{assertion}
\mbox{for}~~ l>n\geq N_0,~~\sigma(u_{2l}, u_{2n+1})-\sigma(u_{2n+1}, \mathcal{G})<\epsilon.
\end{eqnarray}
To prove (\ref{assertion}), assume the contrary. Then there exists $\epsilon_0>0$ such that for each $j\geq 1,$ there exists $l_j>n_j\geq j$ such that $\sigma(u_{2{l_j}-2}, u_{2{n_j}+1})-\sigma(u_{2{n_j}+1}, \mathcal{G}) < \epsilon_0$ and
\begin{eqnarray}
\sigma(u_{2l_j}, u_{2{n_j}+1})-\sigma(u_{2{n_j}+1}, \mathcal{G})&\geq &\epsilon_0. \label{line1}
\end{eqnarray}
\begin{eqnarray*}
\mbox{Then}\qquad \epsilon_0 &\leq & \sigma(u_{2l_j}, u_{2{n_j}+1})-\sigma(u_{2{n_j}+1}, \mathcal{G})\\
 &\leq & \sigma(u_{2l_j}, u_{2{l_j}-2})+ \sigma(u_{2l_j-2}, u_{2{n_j}+1})-\sigma(u_{2{n_j}+1}, \mathcal{G})\\
 &<& \sigma(u_{2l_j}, u_{2{l_j}-2})+\epsilon_0.
\end{eqnarray*} 
By Lemma \ref{Approximationlemma} and property strongly UC, we get $$\displaystyle \lim_{j\to \infty} \left[\sigma(u_{2l_j}, u_{2{n_j}+1})-\sigma(u_{2{n_j}+1}, \mathcal{G})\right]=\epsilon_0.$$ Moreover,
\begin{eqnarray*}
 && \sigma(u_{2l_j}, u_{2{n_j}+1})-\sigma(u_{2{n_j}+1}, \mathcal{G})\\
&&\leq  \sigma(u_{2l_j}, u_{2l_j+2})+ \sigma(u_{2l_j+2}, u_{2n_j+3})+\sigma(u_{2n_j+3}, u_{2n_j+1})-\sigma(u_{2{n_j}+1}, \mathcal{G})\\
&&\leq  \sigma(u_{2l_j}, u_{2l_j+2})+ \sigma(u_{2l_j+1}, u_{2n_j+2})+ \sigma(u_{2n_j+3}, u_{2n_j+1})- \sigma(u_{2{n_j}+1}, \mathcal{G})\\
&&\leq  \sigma(u_{2l_j}, u_{2l_j+2})+\sigma(u_{2l_j}, u_{2n_j+1}) -\psi(\sigma(u_{2l_j}, u_{2n_j+1}))+\psi(\sigma(u_{2n_j+1}, \mathcal{G}))\\
&& \quad +\sigma(u_{2n_j+3}, u_{2n_j+1})-\sigma(u_{2{n_j}+1}, \mathcal{G})
\end{eqnarray*}
As the sequence $\{\sigma(u_{2l_j}, u_{2n_j+1})\}$ is monotone non-increasing and bounded below, Hence it is convergent.
As discussed earlier (see the proof of Lemma \ref{Approximationlemma}), the sequence $\{\sigma(u_{2n_j+1}, \mathcal{G})\}$ converges to $\sigma_0=\displaystyle \lim_{j\to \infty} \sigma(u_{2n_j+1}, u_{2n_j}).$
Letting $j\to \infty,$ we use (\ref{line1}) to get $\epsilon_0\leq 0+\epsilon_0-\psi(\sigma_0+\epsilon_0)+\psi(\sigma_0).$ Thus $\psi(\sigma_0+\epsilon_0)\leq \psi(\sigma_0).$ This is a contradiction to the hypothesis that $\psi$ is strictly increasing. Hence (\ref{assertion}) is proved. By Lemma \ref{Cauchy}, the sequence $\{u_{2n}\}$ is Cauchy. Since $\mathcal{G}$ is complete, there exists $u_0\in \mathcal{G}$ such that $\displaystyle\lim_{n\to \infty} u_{2n}=u_0.$ Using Theorem \ref{cgtsub} we have $\sigma(u_0, Tu_0)=\sigma(Tu_0, \mathcal{G}).$ 
It should be noticed that $\sigma(Tu_0, T^2u_0)=\sigma(Tu_0, \mathcal{G}).$
 By property strongly UC, we have $T^{2}u_0=u_0.$ In fact, $T^{2n}u_0=u_0$ for all $n\geq 1.$

To see the uniqueness of the best approximation in $\mathcal{G},$ let $u_0, z_0\in \mathcal{G}$ with $u_0\neq z_0$ such that 
$\sigma(u_0, Tu_0)=\sigma(Tu_0, \mathcal{G})$ and $\sigma(z_0, Tz_0)=\sigma(Tz_0, \mathcal{G}).$
For $u_0\neq z_0,$ we have $\sigma(z_0, Tu_0)>\sigma(Tu_0, \mathcal{G}).$ Since otherwise by  property strongly UC, we would get $u_0=z_0.$ Now
\begin{eqnarray*}
\sigma(u_0, Tz_0) = \sigma(T^2u_0, Tz_0) &\leq & (I-\psi)(\sigma(z_0, Tu_0))+\psi(\sigma(Tu_0, \mathcal{G}))\\
&< & (I-\psi)(\sigma(z_0, Tu_0))+\psi(\sigma(z_0,Tu_0))\\
&=& \sigma(z_0, Tu_0).
\end{eqnarray*}
We also have $\sigma(u_0, Tz_0)>\sigma(Tz_0, \mathcal{G})$ and using this we have
 \begin{eqnarray*}
\sigma(z_0, Tu_0) = \sigma(T^2z_0, Tu_0) &\leq & (I-\psi)(\sigma(Tz_0, u_0))+\psi(\sigma(Tz_0, \mathcal{G}))\\
&< & (I-\psi)(\sigma(Tz_0, u_0))+\psi(\sigma(u_0,Tz_0))\\
&=& \sigma(u_0, Tz_0).
\end{eqnarray*}
This is absurd. Hence $u_0=z_0.$
\end{proof}
Uniqueness of the best approximation in $\mathcal{G}$ in the above theorem (Theorem \ref{maintheorem}) helps us to find a best approximation in $\mathcal{H}$ as well. In fact, $Tu_0$ is a best approximation of $T$ in $\mathcal{H}$ whenever $u_0$ is a best approximation in $\mathcal{G}.$ 
For any $v\in \mathcal{H},$ we see that $\sigma(Tu_0, \mathcal{G})\leq \sigma(Tu_0, Tv)\leq \sigma(u_0, v).$ Thus $\sigma(Tu_0, \mathcal{G})\leq \sigma(u_0, \mathcal{H}).$ 
This yields $\sigma(T^2u_0, Tu_0)=\sigma(u_0, Tu_0)=\sigma(Tu_0, \mathcal{G})\leq \sigma(u_0, \mathcal{H})=\sigma(T(Tu_0), \mathcal{H}).$ 
 \begin{remark}
 Suppose that $(\mathcal{H},\mathcal{G})$ has property strongly UC in Theorem \ref{maintheorem}.
  In connection with the above discussion, it is worth mentioning that $Tu_0$ is a unique best approximation of $T$ in $\mathcal{H}.$
    To prove this, on the contrary, suppose there is $v\in \mathcal{H}$ such that $v\neq Tu_0$ and $\sigma(v, Tv)=\sigma(Tv, \mathcal{H}).$ Since $\sigma(T^2v, Tv)=\sigma(Tv, \mathcal{H}),$ by property strongly UC of $(\mathcal{H},\mathcal{G}),$ we have $T^{2}v=v.$ Clearly, $\sigma(v,T^2u_0)> \sigma(T^2u_0, \mathcal{H}).$ 
Now we get
 \begin{eqnarray*}
 \sigma(u_0,v)=\sigma(T^4u_0, T^2v) &\leq & (I-\psi)(\sigma(T^3u_0, Tv))+\psi(\sigma(T^3u_0, \mathcal{G}))\\
 &\leq & (I-\psi)(\sigma(T^3u_0, Tv))+\psi(\sigma(T^2u_0, \mathcal{H}))\\
 &<& (I-\psi)(\sigma(T^3u_0, Tv))+\psi(\sigma(T^4u_0, T^2v))\\
 &\leq& \sigma(T^3u_0, T^3v)\\
 &\leq & \sigma(u_0, v).
 \end{eqnarray*} 
 This meets a contradiction. 
 \end{remark}
The following example illustrates Theorem \ref{maintheorem}.
\begin{example}\label{Thmex1}
Consider Example \ref{ex3}. We see that $\mathcal{G}, \mathcal{H}$ and $T$ satisfy all the hypotheses of Theorem \ref{maintheorem}. Thus $T$ has a unique best approximation in $\mathcal{G},$ which is $u=1.$ Moreover, $Tu=-1$ is a unique best approximation for $T$ in $\mathcal{H}.$ Note that $(\mathcal{G},\mathcal{H})$ and $(\mathcal{H},\mathcal{G})$ both have property strongly UC. 
\end{example}
A fixed point result can be obtained.
\begin{theorem}
Let $\mathcal{G}$ and $\mathcal{H}$ be two non-empty subsets of a metric space $(\mathcal{E},\sigma)$ such that $\mathcal{G}$ and $\mathcal{H}$ are complete and $(\mathcal{G},\mathcal{H})$ has property strongly UC. Suppose that $\mathcal{G}\cap \mathcal{H}\neq \emptyset.$ Let $T$ be an almost cyclic $\psi$-contraction on $\mathcal{G}\cup \mathcal{H}.$ For $u_0\in \mathcal{G}\cap \mathcal{H},$ define $u_{n+1}=Tu_n,~n\geq 0.$ 
Then $T$ has a unique fixed point $z$ in $\mathcal{G}\cap \mathcal{H}$ and $\{u_n\}$ converges to $z.$
\end{theorem}
\begin{proof}
We see that $u_{n}\in \mathcal{G}\cap \mathcal{H}$ for all $n\geq 0.$
By Theorem \ref{maintheorem}, $\{u_{2n}\}$ is a Cauchy sequence in $\mathcal{G}\cap \mathcal{H}.$ Then there exists a  $z\in \mathcal{G}\cap \mathcal{H}$ such that $\displaystyle\lim_{n\to \infty} u_{2n}=z.$ Moreover, $\sigma(z,Tz)=\sigma(Tz, \mathcal{G})=0.$ To see the uniqueness of the fixed point, let $y\in \mathcal{G}\cap \mathcal{H}$ with $y\neq z$ such that $y=Ty.$ Then 
\begin{eqnarray*}
\sigma(z,y)= \sigma(Tz,Ty)\leq (I-\psi)(\sigma(z,y))+\psi(\sigma(y,\mathcal{G}))=(I-\psi)(\sigma(z,y))+\psi(0).
\end{eqnarray*} 
This follows $\psi(\sigma(z,y))\leq \psi(0)$ and hence $\sigma(z,y)=0.$
Since $\sigma(u_{n+1}, z)=\sigma(u_{n+1}, Tz)\leq \sigma(u_n, z),$ we say that the sequence $\{d_n\}$ given by $d_n=\sigma(u_n, z),$ $n\geq 1$ is non-increasing. Moreover, $\{d_n\}$ is bounded, and hence convergent. As $d_{2n}\to 0,$ we must say that $\{d_n\}$ converges to $0.$ This completes the proof.
\end{proof}
\begin{corollary}\label{cor:almostcyclic}
 Let $\mathcal{G}$ and $\mathcal{H}$ be two non-empty subsets of a metric space $(\mathcal{E},\sigma)$ such that $\mathcal{G}$ is complete and $(\mathcal{G},\mathcal{H})$ has property strongly UC. Let $T$ be an almost cyclic contraction on $\mathcal{G}\cup \mathcal{H}.$
 Then $T$ has a unique best approximation $u_0$ in $\mathcal{G}$ and $\{T^{2n}u\}$ converges to $u_0$ for every $u$ in $\mathcal{G}.$
\end{corollary}
\begin{corollary}\label{cor:cyclic}
Let $\mathcal{G}$ and $\mathcal{H}$ be two non-empty subsets of a metric space $(\mathcal{E},\sigma)$ such that $\mathcal{G}$ is complete and $(\mathcal{G},\mathcal{H})$ has property strongly UC. Let $T$ be a cyclic contraction on $\mathcal{G}\cup \mathcal{H}.$ Then there exists a unique $u\in \mathcal{G}$ such that $\sigma(u, Tu)=\sigma(\mathcal{G},\mathcal{H}).$
\end{corollary}
\begin{proof}
From Corollary \ref{cor:almostcyclic}, there exists a unique $u\in \mathcal{G}$ such that $\sigma(u,Tu)=\sigma(Tu,\mathcal{G}).$ By the property of cyclic contraction, we have
\begin{eqnarray*}
\|u-Tu\|=\sigma(Tu,\mathcal{G})\leq \|Tu- T^2u\|\leq \beta \|u-Tu\|+(1-\beta) \sigma(\mathcal{G},\mathcal{H}),
\end{eqnarray*}
which forces $\|u-Tu\|=\sigma(\mathcal{G},\mathcal{H}).$  
As $(\mathcal{G},\mathcal{H})$ has property strongly UC, it is easy to see that such a point $u$ is unique.
\end{proof}
We can now clearly see that Theorem \ref{Thm:Eldred2006} and Theorem \ref{Thm:Basha2021} are particular cases of Corollary \ref{cor:cyclic} and Corollary \ref{cor:almostcyclic} respectively.

\section{On almost cyclic contractions}\label{Section4}

In this section, we study continuity property in a metric space and the existence of best approximations in a reflexive Banach space of an almost cyclic contraction. 
It is known that every almost cyclic contraction is not necessarily continuous (\cite{Basha2021}). Here we provide sufficient conditions under which an almost cyclic contraction is continuous.
\begin{theorem}\label{Thm:continuity}
Suppose $(\mathcal{G},\mathcal{H})$ is a non-empty proximal pair of closed convex subsets in a strictly convex Banach space $\mathcal{E}.$ 
Then every relatively nonexpansive map $T$ on $\mathcal{G}\cup \mathcal{H}$ is continuous.  
\end{theorem}
\begin{proof}
Let $\{u_n\}$ be a sequence and $u$ an element in $\mathcal{G}$ such that $u_n\to u.$ Since $(\mathcal{G},\mathcal{H})$ is proximal and $\mathcal{E}$ is strictly convex, there exists a unique $u'$ in $\mathcal{H}$ such that $\|u-u'\|=\sigma(\mathcal{G},\mathcal{H})=\sigma(u', \mathcal{G}).$ It is clear that $\|u_n-u'\|\to \sigma(\mathcal{G},\mathcal{H})$ and using the fact that $T$ is relatively nonexpansive, we have $\|Tu_n-Tu'\|\to \sigma(\mathcal{G},\mathcal{H}).$ 
We claim that $Tu_n\to Tu.$ If not, then there exist an $\epsilon_0>0$ and a subsequence $\{n_j\}$ such that 
$\|Tu_{n_j}-Tu\|\geq \epsilon_0 ~~\mbox{for}~~ j\geq 1.$
Set $H=\{Tu'\}$ and $K=\overline{co}\{Tu_{n_j}, Tu: j\geq 1\}.$ By strict convexity, $\delta(H,K)>\sigma(H,K).$ But $\delta(H,K)=\delta(Tu', K)=\delta\left(Tu', \{Tu_{n_j}, Tu: j\geq 1\}\right)$ and $\sigma(H,K)=\sigma(\mathcal{G},\mathcal{H}).$ There exists a subsequence $\{n_{j_l}\}$ such that $\|Tu'-Tu_{n_{j_l}}\|\to \delta(H,K)$ as $l\to \infty.$ This is a contradiction.
\end{proof}
As a corollary we have the continuity of an almost cyclic contraction:
\begin{theorem}\label{Thm:continuity2}
Suppose $(\mathcal{G},\mathcal{H})$ is a non-empty proximal pair of closed convex subsets in a strictly convex Banach space $\mathcal{E}.$  
Then every almost cyclic contraction $T$ on $\mathcal{G}\cup \mathcal{H}$ is continuous.  
\end{theorem}
It is known (\cite{Raju2015}, Lemma 4.2) that if $(\mathcal{G},\mathcal{H})$ is a pair in a metric space $(\mathcal{E},\sigma)$ such that $\mathcal{G}$ is compact and $\mathcal{H}$ is approximatively compact, then $\mathcal{G}$ and $\mathcal{H}$ are non-empty compact subsets of $\mathcal{E}.$ Thus, the following theorem is obvious.
\begin{theorem}\label{uniform continuity}
Suppose $(\mathcal{G},\mathcal{H})$ is a non-empty proximal pair of convex subsets in a strictly convex Banach space $\mathcal{E}$ such that $\mathcal{G}$ is compact and $\mathcal{H}$ is approximatively compact. Then every almost cyclic contraction $T$ on $\mathcal{G}_0\cup \mathcal{H}_0$ is uniformly continuous.
\end{theorem}
 A metric version of Theorem \ref{Thm:continuity} can also be obtained.
\begin{theorem}\label{Thm:continuity_UC}
Suppose $(\mathcal{G},\mathcal{H})$ is a pair of non-empty closed subsets of a metric space $(\mathcal{E},\sigma)$ such that $(\mathcal{G},\mathcal{H})$ and $(\mathcal{H},\mathcal{G})$ has property UC. Then every almost cyclic contraction $T$ on $\mathcal{G}\cup \mathcal{H}$ is continuous.
\end{theorem}
\begin{proof}
Let $\{u_n\}$ be a sequence in $\mathcal{G}$ such that it converges to $u$ in $\mathcal{G}.$ Since $T$ is almost cyclic contraction, it is relatively nonexpansive and hence
\begin{eqnarray*}\label{eqn:continuity}
\sigma(Tu_n, Tu')\leq \sigma(u_n, u')\leq \sigma(u_n, u)+\sigma(u,u')\to \sigma(\mathcal{G},\mathcal{H}).
\end{eqnarray*} 
But $\sigma(Tu, Tu')=\sigma(\mathcal{G},\mathcal{H}).$ By property UC of $(\mathcal{H},\mathcal{G})$, we get $\sigma(Tu_n,Tu)$ $\to 0. $
\end{proof}

The following result gives a flavour to establish the best approximation results for almost cyclic contractions in the setting of a reflexive Banach space.
\begin{theorem}\label{Additional1}
Suppose that $\mathcal{G}$ and $\mathcal{H}$ are two non-empty subsets of a reflexive Banach space $\mathcal{E}$ such that $\mathcal{G}$ is weakly closed and $\mathcal{H}$ is compact. Assume that $T:\mathcal{G}\cup \mathcal{H}\to \mathcal{G}\cup \mathcal{H}$ is an almost cyclic contraction.
Then there exists $(u, v)\in \mathcal{G}\times \mathcal{H}$ such that $\|u-v\|=\sigma(v, \mathcal{G}).$ Moreover, for any $u_0\in \mathcal{G}$ and $u_{n+1}=Tu_n,~n\geq 0,~x=\lim_{n\to \infty} u_{2n}$ and $v= \lim_{n\to \infty} u_{2n+1}.$
\end{theorem}
\begin{proof}
For $u_0\in \mathcal{G},$ define $u_{n+1}=Tu_n,~n\geq 0.$ By Proposition 3.3 of \cite{Basha2021}, the sequence $\{u_n\}$ is bounded. Then there exists $u\in \mathcal{G}$ and $v\in \mathcal{H}$ such that $u_{2n_k}\to u$ weakly and $u_{2n_k+1}\to v$. Thus $\sigma(u_{2n_k+1}, \mathcal{G})\to \sigma(v, \mathcal{G})$ and hence $\|u_{2n_k}-u_{2n_k+1}\|\to \sigma(v, \mathcal{G}).$ Also, $u_{2n_k}-u_{2n_k+1}\to u-v\neq 0$ weakly. By Hahn-Banach thereom, there exists $u^*\in \mathcal{E}^*$ such that $\|u^*\|=1$ and $u^*(u-v)=\|u-v\|.$ Since
$u^*(u_{2n_k}-u_{2n_k+1})\to u^*(u-v)=\|u-v\|,$ we have
\begin{eqnarray*}
\|u-v\|=\lim_{k\to \infty} |u^*(u_{2n_k}-u_{2n_k+1})|\leq \lim_{k\to \infty} \|u_{2n_k}-u_{2n_k+1}\|= \sigma(v,\mathcal{G}). 
\end{eqnarray*}
This proves that $\|u-v\|=\sigma(v, \mathcal{G}).$
\end{proof} 
The following result easily follows from Theorem \ref{Additional1}.
\begin{theorem}\label{Additional2}
Suppose $\mathcal{G}$ and $\mathcal{H}$ are two non-empty subsets of a reflexive Banach space $\mathcal{E}$ such that $\mathcal{G}$ is weakly closed, $\mathcal{H}$ is compact and $T$ is almost cyclic contraction on $\mathcal{G}\cup \mathcal{H}.$ Assume that $T$ is weakly continuous on $\mathcal{G}.$ Then there exists $u\in \mathcal{G}$ such that $\|u-Tu\|=\sigma(Tu, \mathcal{G}).$  
\end{theorem}
\begin{proof}
Let $u_0\in \mathcal{G}$ and for $n\geq 0,$ define $u_{n+1}= Tu_n.$ Then $\{u_{2n}\}$ is bounded. There is a subsequence $\{n_j\}$ such that $u_{2n_j}\to u$ weakly, for some $u\in \mathcal{G}.$ Since $T$ is weakly continuous, we get $Tu_{2n_j}\to Tu$ weakly. Therefore, $u_{2n_j}-Tu_{2n_j}\to u-Tu$ weakly. By Theorem \ref{Additional1}, $\|u-Tu\|=\sigma(Tu, \mathcal{G}).$
\end{proof} 
For two non-void subsets $\mathcal{G}$ and $\mathcal{H}$ of a normed linear space $\mathcal{E}$ and a cyclic map $T:\mathcal{G}\cup \mathcal{H}\to \mathcal{G}\cup \mathcal{H},$ we say that $T$ has {\it{proximinal property}} if $\{u_n\}$ is a sequence in $\mathcal{G}$ (resp. in $\mathcal{H}$) which converges weakly, say to $u,$ for some $u$ in $\mathcal{G}$ (resp. in $\mathcal{H}$) and $\|u_n-Tu_n\|- \sigma(Tu_n, \mathcal{G})\to 0$ (resp. $\|u_n-Tu_n\|- \sigma(Tu_n, \mathcal{H})\to 0$), then it is the case that $\|u-Tu\|=\sigma(Tu,\mathcal{G})$ (resp., $\|u-Tu\|=\sigma(Tu,\mathcal{H})$). Note that if $\mathcal{G}=\mathcal{H},$ then {\it{proximinal property}} boils down to the demiclosedness of the map $I-T$ at the point $0.$
\begin{theorem}\label{Additional3}
Suppose $\mathcal{G}$ and $\mathcal{H}$ are two non-void subsets of a reflexive Banach space $\mathcal{E}$ such that $\mathcal{G}$ is weakly closed and $T$ is an almost cyclic contraction on $\mathcal{G}\cup \mathcal{H}.$ Assume that $T$ has {\it{proximinal property}} on $\mathcal{G}\cup \mathcal{H}.$ Then there exists $u\in \mathcal{G}$ such that $\|u-Tu\|=\sigma(Tu, \mathcal{G}).$
\end{theorem}
\begin{proof}
Let $u_0\in \mathcal{G}$ and define $u_{n+1}=Tu_n,~n\geq 0.$ By Proposition 3.3 of \cite{Basha2021}, $\{u_{2n}\}$ and $\{u_{2n+1}\}$ are bounded. Since $\mathcal{G}$ is weakly closed and $\mathcal{E}$ is reflexive, there exist a subsequence $\{u_{2n_k}\}$ and $u$ in $\mathcal{G}$ such that $u_{2n_k}\to u$ weakly. By Lemma \ref{Approximationlemma}, we have $\|u_{2n_k}-u_{2n_k+1}\|-\sigma(u_{2n_k+1}, \mathcal{G})\to 0.$ Using proximinal property, we get $\sigma(u, Tu)=\sigma(Tu, \mathcal{G}).$
\end{proof}

\begin{remark}\label{remark1}
If we assume $\mathcal{E}$ is strictly convex and $\mathcal{G}$ is convex in Theorem \ref{Additional2}, then no $z\in \mathcal{G}$ with $z\neq u$ satisfies $\|z-Tu\|=\sigma(Tu, \mathcal{G}).$  For, if it does, by strict convexity, we have $\left\|Tu-\frac{u+z}{2} \right\|=\left\|\frac{u-Tu}{2}+\frac{z-Tu}{2} \right\|<\sigma(Tu, \mathcal{G}),$ which is absurd.
\end{remark}

It is interesting to investigate the existence and uniqueness of best approximation of an almost cyclic $\psi$-contraction in a reflexive Banach space by dropping the assumption that $\psi$ is continuous.

%

\end{document}